\UseRawInputEncoding
\documentclass[12pt, reqno]{amsart}
\usepackage{amssymb}
\usepackage{amsfonts}
\usepackage{amssymb,amsmath,amscd}
\usepackage[colorlinks=true]{hyperref}
\newtheorem{Theorem}{\indent Theorem}[section]

\newtheorem{Lemma}[Theorem]{\indent Lemma}

\theoremstyle{remark}
\newtheorem{Remark}{Remark}

\begin{document}
\centerline{
\bf A note on simple zeros related to  Dedekind zeta functions
}
\bigskip
\centerline{Wei Zhang}

\textbf{Abstract}
We  give a conditional lower bound on  the number of non-trivial simple zeros for the Dedekind zeta function $\zeta_{K}(s)$, where $K$ is a quadratic number field. The conditional result is given by assuming a  Lindel\"of on average  (in the $L^{6}$ sense) for both $\zeta(s)$ and $L(s,\chi)$, which can be seen as a stronger version of Conrey-Gonek-Ghosh's \cite{c} conditional result. This improves upon the work of Wu and Zhao \cite{Zhao}, who had a similar
result.
%, where they assumed instead a ``Lindel\"of-on-average" in the $L^{8}$ sense, rather than the $L^{6}$ sense.

\textbf{Keywords:} Simple zeros,   zeta function.

\textbf{2000 Mathematics Subject Classification:}11M06, 11M41.

\bigskip
\bigskip
\numberwithin{equation}{section}

\section{Introduction}
A folklore conjecture in analytic number theory is that for any number
field $K$, almost all zeros of $\zeta_{K}(s)$ in the critical strip are simple. This
conjecture is attractive because of how hard it is to prove to the best of
my knowledge of the existing literature, we still do not know if $\zeta_{K}(s)$ has infinitely many simple
zeros if $[K:Q]\geq3$ and it was not known if even the simplest case of
$\zeta(s)$ (Riemann zeta function) has infinitely many simple zeros in the critical strip until it was
observed independently by Heath-Brown \cite{Hea} and Selberg \cite{Sel} that Levinson's
method \cite{Lev}
yields a positive proportion of simple zeros on the
critical line.

In this note, we are interested in  the non-real zeros of $\zeta_{K}(s)$ in the critical strip
\[
\mathcal{R}=\{s=\sigma+it:0<\sigma<1,0<t<T\}.
\]
Let $K$ be a certain quadratic extension of the $\mathbb{Q}$ and let $\zeta_{K}(s)$ be the corresponding Dedekind zeta function. Then we have
\[
\zeta_{K}(s)=\sum_{\mathfrak{a}}\frac{1}
{(\mathfrak{Na})^{s}}
=\sum_{n=1}^{\infty}\frac{a_{K}(n)}{n^{s}},\ \ \Re(s)>1,
\]
where $\mathfrak{a}$ suns over the non-zero integral ideals of $K,$  $\mathfrak{Na}$ is the norm of $\mathfrak{a},$ and $a_{K}(n)=\#\{\mathfrak{a}:\mathfrak{Na}=n\}$ is the ideal counting function of $K.$
It is well known that
\[
N_{K}(T)=\#\{\rho_{K}: \zeta_{K}(\rho_{K})=0,\rho_{K}\in\mathcal{R}\}
\sim \pi^{-1} T\log T.
\]
Moreover, it is  also generally believed that
all the non-real zeros of $\zeta_{K}(s)$ in $\mathcal{R}$ are also simple zeros. However, one even cannot confirm that whether there are
infinitely many  non-real simple zeros in $\mathcal{R}$ before 1986.
In 1986, Conrey, Ghosh and Gonek \cite{c} introduced
a new method to deal with the simple zeros for such type problems.
For sufficiently large $T,$ Conrey, Ghosh and Gonek \cite{c} proved that
\[
N'_{K}(T)\gg T^{\frac{6}{11}},
\]
where
\[
N'_{K}(T)=\#\{\rho_{K}: \zeta_{K}(\rho_{K})=0,\zeta'_{K}(\rho_{K})\neq0,\rho_{K}\in\mathcal{R}\}
.
\]
This implies that there are  infinitely many simple zeros for the $\zeta_{K}(s)$ associated to the
quadratic number fields. The exponent in this result depends upon progress towards the
Lindel\"{o}f hypothesis.   In fact, a better result is also given in \cite{c}.
Precisely, for any $\varepsilon>0,$ sufficiently large $T$ and
\[
\theta=\max\left\{\frac{1}{1+6c},\frac{\sqrt{1+16c+16c^{2}}-1-4c}
{4c}\right\},
\]
it is proved in \cite{c} that
\[N'_{K}(T) \gg T^{\theta-\varepsilon},
\]
where $c$ is determined by the subconvexity $\zeta(1/2+it)\ll(|t|+1)^{c+\varepsilon}.$
One can find that the Lindel\"{o}f hypothesis implies that
\[N'_{K}(T) \gg T^{1-\varepsilon},
\]
which means that one can get almost to a positive proportion.
It is worth emphasizing that the zeros what they find are also simple zeros of $ \zeta(s),$
where $\zeta(s)$ is the well known Riemann zeta function.
In another paper \cite{cc},  the authors also proved that the Riemann Hypothesis implies that more than $1/54$ of the non-real zeros of $\zeta_{K}(s)$ are simple.

Very recently, in \cite{Zhao}, some  results in the above are improved by showing (unconditionally) that
\[N'_{K}(T) \gg T^{\frac67-\varepsilon}.
\]
In \cite{Zhao}, some conditional results were also given.
To state these conditional results precisely, we need introduce two conjectures.

 {\bf Conjecture 1.}\label{zero1}
Let $N(\sigma,T)$ denote the number of zeros of $\zeta(s)$ for $\beta\geq\sigma$ and $|\gamma|\leq T.$ Then for any $\varepsilon>0$ and $1/2\leq\sigma\leq1,$ we have
\[
N(\sigma,T)\ll T^{A(\sigma)(1-\sigma)+\varepsilon}.
\]

 {\bf Conjecture 2.}\label{k1}
Let $q\geq1$ be a fixed integer, and let $\chi$ be a
real primitive character to the modulus $q$. For $\bar{s}=1/2+it,$ any $\varepsilon>0,$ any real number $k>4$ and sufficiently $T,$ we have
\[
\int_{1}^{T}\left|L(\bar{s},\chi)\right|^{k}\ \mathrm{d}t\ll_{q} T^{1+\varepsilon},
\]
where
$
L(s,\chi)=\sum_{n=1}^{\infty}\chi(n)n^{-s}
$ $(\Re s>1)$
is the Dirichlet $L$-function.

In \cite{Zhao}, Wu and Zhao also proved that
\begin{itemize}
\item If Conjecture 1 is true with $A(\sigma)=2$, then
\[N'_{K}(T) \gg T^{1-\varepsilon}.
\]
\item
If Conjecture 2 is true with $k=\bf{8}$, then
\[N'_{K}(T) \gg T^{1-\varepsilon}.
\]
\end{itemize}
Surprisingly, we find that
  $k=\bf{6}$ for Conjecture 2 is sufficient enough for this purpose.
We also remark that in this method, only with a real number $k\in(4,6),$ one cannot get the desired conclusion. Now we intend to give a concrete reason for this.
We assume that Conjecture 2 is true with certain $k$ and assume that
\[
N(\sigma,T) \ll T^{(2l+2)(1-\sigma)/(l+2-2\sigma)+\varepsilon},
\]
which is from the result of \cite{Zh}.
Then in the process of this paper, we can obtain that
\[
S(T,1)-S(T,a)\ll
T^{1-B(\sigma,k,l)(\sigma-1/2)+(1-2/k)\varepsilon},
\]
where
\[
B(\sigma,k,l)=2l(1-2/k)/(l+2-2\sigma)-1.
\]
Then we have
\[
B(\sigma,k,l)=\frac{2l(1-2/k)-l-2+2\sigma}{l+2-2\sigma}.
\]
We need $2l(1-2/k)-l-2+2\sigma>0.$ Hence for $1/2+\varepsilon<\sigma<1,$ we have
\[
\frac{k}{k-4}< l+\varepsilon.
\]
When we choose $k={\bf 6}$ in Conjecture 1, Conjecture 1 just implies the density result we needed of $l=3,\theta=1,\alpha=1$ in \cite{Zh} (see also in \cite{i}).

Moreover, one can also give conditional results for this problem by assuming Conjecture 2 ($k\in(4,6)$) with  assumed results that are weaker than the density hypothesis but stronger than
\[
N(\sigma,T)\ll T^{A(\sigma)(1-\sigma)+\varepsilon},
\]
where $A(\sigma)=8/(5-2\sigma).$
\begin{Theorem}\label{sm}
If Conjecture 2 is true with $k=\bf{6}$, then
\[N'_{K}(T) \gg T^{1-\varepsilon}.
\]
\end{Theorem}

\begin{Remark}
It is worth remarking that if the $L^{6}$-Lindel\"{o}f on average
  could be replaced by an $L^{4}$-Lindel\"{o}f on average, then the theorem would become
unconditional. Since we have known how to prove $L^{4}$-Lindel\"{o}f-on-average
for Dirichlet $L$-functions by the classical work of Ingham on the
fourth moment of $\zeta(s)$ (see Titchmarsh \cite{Tit} or \cite{i}).
\end{Remark}

\section{Preliminaries}
For proof of Theorem \ref{sm}, the method in this paper relies heavily on the relation
\[
\zeta_{K}(s)=\zeta(s)L(s,\chi),
\]
where $\chi$ is a real primitive character to the modulus $q=|D|$, where $D$ is   the discriminant of $K.$
It is well known that
\[
\zeta(1-s)=2(2\pi)^{-s}\Gamma(s)\cos(\pi s/2)\zeta(s),
\]
and
\begin{align*}
&L(1-s,\chi)
\\&=\frac{q^{s}}{\sum_{l=1}^{q}\chi(l)e(l/q)}(2\pi)^{-s}\Gamma(s)(e^{\pi is/2}+\chi(-1)e^{-\pi is/2})L(s,\chi),
\end{align*}
and
\[
\zeta'_{K}(s)=\zeta'(s)L(s,\chi)+\zeta(s)L'(s,\chi).
\]
Let $\rho=\beta+i\gamma$ denote the non-trivial zeros of   $\zeta(s)$.
Then it is implied that (for instance, see (8) in \cite{c} and (2.6)-(2.8) in \cite{Zhao})
\begin{align}\label{fe}
|L(1-\rho,\chi)|\ll T^{ \beta-1/2}|L(\rho,\chi)|
\end{align}
and
\begin{align}\label{fed}
|\zeta'(1-\rho )|\ll T^{ \beta-1/2}|\zeta'(\rho )|.
\end{align}
For $1/2\leq a\leq1,$ let
\[
S(T,a)=\sum_{\rho\in\mathcal{R}, 1-a\leq \beta \leq a }\zeta'(\rho)L(1-\rho,\chi).
\]
Moreover, let
\[
N'_{K}(T,a)=\#\{\rho_{K}: \rho_{K}=\beta_{K}+i\gamma_{K}\in \mathcal{R},\rho_{K}\ \textup{is\ a\ simple\ zero\ of}\ \zeta_{K}\ \textup{and}\ 1/2\leq \beta_{K} \leq a \}.
\]
Then we can introduce the following lemmas.
\begin{Lemma}\label{cz}
For some $1/2\leq a\leq5/8,$ suppose that $|S(T,a)|\gg T.$ Then
\[
N'_{K}(T,a)\gg T^{\frac{2a}{4a-1}-\varepsilon}.
\]
\end{Lemma}
\begin{proof}
See Lemma 8 in \cite{Zhao}.
\end{proof}
\begin{Lemma}\label{c}
With the notations above, we have
\[
S(T,1)=L(1,\chi)\frac{T}{4\pi}\log^{2}T+O(T\log T).
\]
\end{Lemma}
\begin{proof}
This is (18) in \cite{c}.
\end{proof}

\begin{Lemma}\label{i00}
For $k\geq1$ a fixed integer let $1/2\leq \sigma_{k}^{*}<1$ denote the infimum of all numbers $\sigma_{k}^{*}$ for which
\[
\int_{1}^{T}|\zeta(\sigma_{k}^{*}+it)|^{2k}dt\ll T^{1+\varepsilon}
\]
holds for any $\varepsilon>0.$ Then the asymptotic formula
\[
\int_{1}^{T}|\zeta(\sigma+it)|^{2k}dt=T\sum_{n=1}^{\infty}
d_{k}(n)^{2}n^{-2\sigma}+R(k,\sigma,T)
\]
holds for $\sigma_{k}^{*}<\sigma<1$ with
\[
R(k,\sigma,T)\ll T^{(2-\sigma-\sigma_{k}^{*})/(2-2\sigma_{k}^{*})}.
\]
\end{Lemma}
\begin{proof}
See Lemma 8.4 in \cite{i}.
\end{proof}

\begin{Lemma}\label{z}
If Conjecture 2 is true with $k={\bf6},$
then for $1/2\leq\sigma\leq1,$ we have
\begin{align}\label{zd}
N(\sigma,T)\ll T^{\frac{8-8\sigma}{5-2\sigma}+\varepsilon},
\end{align}
\begin{align}\label{md}
\sum_{\substack{\rho\in\mathcal{R} \\ \beta \geq\sigma}}
|\zeta'(\rho)|^{6}\ll T^{1+\varepsilon},
\end{align}
and
\begin{align}\label{md1}
\sum_{\substack{\rho\in\mathcal{R} \\ \beta \geq\sigma}}
|L(\rho,\chi)|^{6}\ll T^{1+\varepsilon}.
\end{align}
\end{Lemma}
\begin{proof}
(\ref{zd}) can be obtained by Theorem 1 in \cite{Zh} for the case $l=3,$ $\alpha=1,$ and $\theta=1.$ (see also in \cite{i}).
Now we begin to prove (\ref{md}).
By Cauchy's theorem,
\[
\zeta'(\rho)=\frac{1}{\Delta}\int_{\Delta}^{2\Delta}
\frac{1}{2\pi i}\int_{|s-\rho|=\delta}
\frac{\zeta(s)}{(s-\rho)^{2}}dsd\delta,
\]
where $\Delta=(\log T)^{-1}.$ Hence if $1/2\leq \beta <1,$ then using H\"older's inequality we have
\[
|\zeta'(\rho)|^{6}\ll T^{\varepsilon}\int_{1/2-2\Delta}^{1+2\Delta}
\int_{1}^{T+1}|\zeta(\sigma+it)|^{6}dtd\sigma.
\]
Since the number of $\rho$ in a square of side length 1 is $\ll (\log T),$ then by Lemma \ref{i00} we have
\[
\sum_{\substack{\rho\in\mathcal{R}
\\ \beta\geq1/2}}|\zeta'(\rho)|^{6}
\ll T^{\varepsilon}\int_{1/2-2\Delta}^{1+2\Delta}\int_{1}^{T+1}
|\zeta(\sigma+it)|^{6}dt d\delta \ll
T^{1+\varepsilon}.
\]
(\ref{md1}) can be proved by adapting the argument of \cite{c,Zhao} and Lemma \ref{i00} in assuming Conjecture 2 with $k=6.$
\end{proof}

\section{Proof of  Theorem \ref{sm}}
 Then we will prove the following lemma.
\begin{Lemma}\label{w}
For any sufficiently small $\varepsilon>0$ and  $a\geq1/2+10\sqrt{\varepsilon},$ if Conjecture 2 is true with $k={\bf6},$
we have
\[
S(T,a)\asymp T\log^{2}T.
\]
\end{Lemma}
\begin{proof}

Recall that
\[
S(T,1)-S(T,a)\ll \sum_{\substack{\rho\in\mathcal{R} \\ \beta \geq a}}
|\zeta'(\rho)L(1-\rho,\chi)|+\sum_{\substack{\rho\in\mathcal{R} \\ \beta \geq a}}
|\zeta'(1-\rho)L(\rho,\chi)|.
\]
Then by (\ref{fe}) and (\ref{fed}), one has
\begin{align}\label{000}
S(T,1)-S(T,a)
\ll \sum_{\substack{\rho\in\mathcal{R}\\\beta \geq a}}T^{\beta-\frac12}
|\zeta'(\rho)L(\rho,\chi)|.
\end{align}
By using the Riemann-Stieltjes integral (or the partial summation), one can obtain that
\begin{align*}
\sum_{\substack{\rho\in\mathcal{R}\\\beta \geq a}}T^{\beta-\frac12}
|\zeta'(\rho)L(\rho,\chi)|&=\int_{a}^{1}T^{\sigma-1/2}d
F(\sigma,\mathcal{R}),
\end{align*}
where
\[
F(\sigma,\mathcal{R}):=\sum_{\substack{\rho\in\mathcal{R} \\ \beta \geq\sigma}}
|\zeta'(\rho)L(\rho,\chi)|.
\]
Then applying  integration by parts, we have
\begin{align*}
\sum_{\substack{\rho\in\mathcal{R}\\\beta \geq a}}T^{\beta-\frac12}
|\zeta'(\rho)L(\rho,\chi)|
&\ll T^{1-\frac12}
F(1,\mathcal{R})
-T^{a-\frac12}
F(a,\mathcal{R})\\
&-\int_{a}^{1}T^{\sigma-1/2}F(\sigma,\mathcal{R})(\log T)
d\sigma.
\end{align*}

This gives that
\begin{align*}
\sum_{\substack{\rho\in\mathcal{R}\\\beta \geq a}}T^{\beta-\frac12}
|\zeta'(\rho)L(\rho,\chi)|
&\ll \max _{a\leq \sigma\leq1}T^{\sigma-\frac12}\sum_{\substack{\rho\in\mathcal{R} \\ \beta \geq\sigma}}|\zeta'(\rho)L(\rho,\chi)|(\log T)\\
&\ll \max _{a\leq \sigma\leq1}T^{\sigma-\frac12+\varepsilon}\sum_{\substack{\rho\in\mathcal{R} \\ \beta \geq\sigma}}|\zeta'(\rho)L(\rho,\chi)|.
\end{align*}

By H\"{older}'s inequalities, we have
\[
\sum_{\substack{\rho\in\mathcal{R} \\ \beta \geq\sigma}}
|\zeta'(\rho)L(\rho,\chi)|\leq
N(\sigma,T)^{\frac23}\left(\sum_{\substack{\rho\in\mathcal{R} \\ \beta \geq\sigma}}
|\zeta'(\rho)|^{6}\right)^{\frac16}\left(\sum_{\substack{\rho\in\mathcal{R} \\ \beta \geq\sigma}}
|L(\rho,\chi)|^{6}\right)^{\frac16}.
\]
Recall Lemma \ref{z} and the 6th moment hypothesis, one can get
\begin{align*}
\sum_{\substack{\rho\in\mathcal{R}\\\beta \geq\sigma}}
|\zeta'(\rho)L(\rho,\chi)|\ll
T^{\frac{8-8\sigma}{5-2\sigma}\times\frac23+\frac13+ \varepsilon}.
\end{align*}
Then
\begin{align*}
\sum_{\substack{\rho\in\mathcal{R}\\\beta \geq\sigma}}
|\zeta'(\rho)L(\rho,\chi)|\ll T^{1+\frac{3-6\sigma}{5-2\sigma}\times\frac23+ \varepsilon}
\ll
T^{1-\frac{4}{5-2\sigma}\times(\sigma-\frac12)+ \varepsilon}
.\end{align*}
Hence by (\ref{000}), we have
\begin{align*}
S(T,1)-S(T,a)&\ll \max _{a\leq \sigma\leq1}
T^{\sigma-1/2}T^{1-\frac{4}{5-2\sigma}\times(\sigma-\frac12)+ \varepsilon}\\&\ll \max _{a\leq \sigma\leq1} T^{1-\frac{4}{5-2\sigma}\times(\sigma-\frac12)+\frac{5-2\sigma}{5-2\sigma}\times(\sigma-\frac12)+ \varepsilon}
.\end{align*}
This gives that
\[
S(T,1)-S(T,a)\ll  \max _{a\leq \sigma\leq1} T^{1-\frac{2}{5-2\sigma}\times(\sigma-\frac12)^{2}+ \varepsilon}
.\]
For suitable $\varepsilon'=10\sqrt{\varepsilon},$ $a\geq1/2+\varepsilon'$ and $a\leq \sigma\leq 1,$ we can get that
\begin{align*}
S(T,1)-S(T,a)&\ll T^{1-\frac{25}{1-5\sqrt{\varepsilon}}+\varepsilon}\\
&\ll T^{1-24\varepsilon}\\
&\ll T^{1-\varepsilon}.
\end{align*}
Recall (see  Lemma \ref{c}) that
\[
S(T,1)=L(1,\chi)\frac{T}{4\pi}\log^{2}T+O(T\log T).
\]
Hence  for $a\geq1/2+\varepsilon',$ we can obtain that
\[
S(T,a)\asymp T\log^{2}T .
\]
This completes the proof of Lemma \ref{w}.
\end{proof}
Now our theorem can be deduced by Lemma \ref{cz}, Lemma \ref{w}
 and the choice of $a=1/2+\varepsilon'.$

\vspace{0.5cm}

$\mathbf{Acknowledgements}$
I am deeply grateful to the referee(s) for carefully reading the manuscript and making useful suggestions.

\address{Wei Zhang\\ School of Mathematics and Statistics\\
               Henan University\\
               Kaifeng  475004, Henan\\
               China}
\email{zhangweimath@126.com}
\end{document}